\documentclass{article}%
\usepackage{amsmath,amsfonts,amssymb,graphicx}%

\newtheorem{thm}{Theorem}
\newtheorem{lem}[thm]{Lemma}

\newtheorem{pro}[thm]{Proposition}

\def\a{\alpha}
\def\b{\beta}
\def\g{\gamma}
\def\P{\mathcal{P}}
\def\H{\mathcal{H}}
\def\Bbb{\mathbb}
\def\rk{{\rm rank}}
\def\proof{{\noindent \bf Proof. \hspace{.01in}}}
\newcommand{\qed}{\hspace{.1in} \vrule height 7pt width 5pt depth 0pt \medskip}

\begin{document}


\title{\Large Enumerating extreme points of
the polytopes of stochastic tensors: an optimization approach}
\author{
Fuzhen Zhang${}^{\,\rm a}$
,\;
Xiao-Dong Zhang${}^{\,\rm b}$
\\
\footnotesize{${}^{\,\rm a}$ Nova Southeastern University, Fort Lauderdale, USA; zhang@nova.edu}\\
\footnotesize{${}^{\,\rm b}$ Shanghai Jiao Tong University, Shanghai, China;  xiaodong@sjtu.edu.cn}}

\date{}
\maketitle

\bigskip
 \hrule
\bigskip

\noindent {\bf Abstract.}
This paper is concerned with the extreme points of
the polytopes of stochastic tensors.
By a tensor we mean a multi-dimensional array over the real number field.
A line-stochastic tensor is a nonnegative tensor in which the sum of all entries on each line (i.e., 1 free index)  is equal  to 1; a plane-stochastic tensor is a nonnegative tensor in which the sum of all entries on each plane (i.e., 2 free indices)  is equal to 1.
In enumerating extreme points  of the  polytopes of line- and plane-stochastic tensors of order 3 and dimension $n$, we consider the approach by linear optimization  and present new lower and upper bounds.
We also study the coefficient matrices that define the polytopes.

\medskip


\bigskip
\hrule

\date{}
\maketitle


\section{Introduction}

Determination of extreme points of a convex set (in certain space) is basic and of central importance in convex analysis (see, e.g., \cite[Chapter 4]{Web94}) and operations research (see, e.g., \cite[Chapter 7]{MarDover13}) due to the fact that objective functions to be studied with constraints are usually defined over convex sets. When the convex set, say, a polytope,  is defined via a system of linear (in)equalities, an approach of linear programming or optimization  plays an important role in finding the extreme points (or vertices)  by  finding the optimal (i.e., basic) solutions of the system.
Existence, structure and enumeration of extreme points are thematic.
 This paper is concerned with enumeration of extreme points
 of some polytopes.

We begin with some basic terminologies  in the theory of polytopes.

A {\em polytope} $\P \subset \Bbb R^d$ is  the convex hull of a finite set of points
in $\Bbb R^d$ (see, e.g., \cite[p.\,8]{Bar02} or \cite[p.\,4]{Zie95}).
  Equivalently,
   a polytope is a bounded intersection of finitely many  closed  halfspaces
    in the form $\{x\in \Bbb R^d\mid Ax\leq b\}\subset \Bbb R^d$ (\cite[p.\,29]{Zie95}),
    where $A$ is an $m\times d$ real matrix for some positive integer $m$ and $b\in \Bbb R^m$.
   The {\em dimension} of the polytope $\P$ is the minimum of dimensions of all affine spaces that contain $\P$. It turns out  that the dimension of $\P$ is equal to the dimension of the null space of $A$; that is, $\dim (\P)=d-\rk (A)$.

Recall the   Birkhoff polytope $ \mathcal{B}_n$ of $n\times n$ doubly stochastic matrices. The celebrated  Birkhoff-von Neumann Theorem states that $\mathcal{B}_n$ is the convex hull of all $n\times n$  permutation matrices (see, e.g., \cite[p.\,159]{ZFZbook11}).
As a polytope in $\Bbb R^{n^2}\hspace{-.05in},$    $\;\mathcal{B}_n$ has dimension $(n-1)^2$,  $n^2$ facets, and $n!$ vertices (see, e.g., \cite[p.\,20]{Zie95}).  Carath\'{e}odory's theorem ensures that
every  $n\times n$ doubly stochastic matrix can be written as a convex combination of at most $n^2-2n+2$ permutation matrices. 
Geometrical and combinatorial properties of  the Birkhoff polytope  have been extensively studied; see, e.g.,  \cite{BrCao, BrCs75, CLN14, Paf15};
see also \cite[pp.\,47--52]{MOA11} for a brief account on the topic.


We call a  hyperplane $\H$ {\em supporting hyperplane} of a polytope $\P$ if $\H\cap \P\not = \emptyset$ and if  $\P$ is contained in one of the two closed
halfspaces bounded by $\H$.  The intersection $\mathcal{F}=\H\cap \P$ is a {\em face} of $\P$. A {face} of  dimension 0 is called a {\em vertex} or an {\em extreme point} of $\P$;
a {\em facet} of $\P$ is a face of dimension $\dim (\P)-1$. Determining  the $f$-vector for a polytope,
i.e., $f(\P)=(f_0, f_1, \dots, f_{d-1})$, where $f_i$ is the number of faces of dimension $i$,  is
an uneasy task in the combinatorial theory of convex polytopes \cite[p.\,66]{Bro83}.
Special attention has been paid to $f_0(\P)$, the number of vertices (as well as $f_{d-1}(\P)$, the number of facets).

By the Krein-Milman Theorem (see, e.g., \cite[p.\,121]{Bar02}),  every polytope
is the convex hull of its vertices (extreme points). It is a fundamental and central  question in the polytope theory to find the number and structures of the vertices (or faces) for a given polytope.
This is  an extremely difficult problem in general as
  the  McMullen Upper Bound Theorem  shows (see, e.g., \cite{Bro83} or \cite[p.\,254]{Zie95}).

We are concerned with the polytopes of stochastic tensors in this paper.
By a {\em  tensor} we mean a multidimensional array (i.e., matrix of higher order or hypermatrix) over the real number field $\Bbb R$.
Let ${n_1, n_2, \dots, n_d}$ be positive integers. We write $A=(a_{i_1i_2\dots i_d})$,
$i_k=1, 2,  \dots, n_k$, $k=1, 2,  \dots, d$,  for an $n_1\times n_2\times  \cdots \times n_d$  tensor   of order $d$ (the number of indices).
The tensors of order 1 (i.e., $d=1$) are vectors in $\Bbb R^{n_1}$, while the 2nd order tensors are just the regular $n_1\times n_2$ matrices. A 3rd order tensor, namely, an $n_1\times n_2\times n_3$ tensor,  may be viewed as a {book} of  $n_3$ pages (slices), each page is
an $n_1\times n_2$ matrix.

If $n_1= n_2=\dots= n_d=n$, we say that
$A$ is of order $d$ and dimension $n$  or  $A$ is an $\overbrace{n\times   \cdots \times n}^d$ tensor (also referred to as
 {\em tensor cube} or
 {\em 3D matrix}; see \cite{WZ17Per}).

 For a nonnegative tensor $A=(a_{i_1i_2\dots i_d})$
of order $d$ and dimension $n$, we say that $A$ is {\em line-stochastic} \cite{FiSw85}
  if the sum of  the entries on each line is   1:
$$\sum_{i=1}^n a_{\cdots i\cdots }=1$$
and  $A$ is {\em plane-stochastic} \cite{BrCs75laa} if the sum of all elements on every plane is 1:
$$\sum_{i, j=1}^n a_{\cdots i\cdots j \cdots}=1.$$

 We denote by $\mathcal{L}_n$ the polytope of  the $n\times n\times n$ (triply) line-stochastic tensors
and by $\mathcal{P}_n$ the $n\times n\times n$ (triply) plane-stochastic tensors.

More generally, let $A$ be a nonnegative tensor of order $d$ and dimension $n$ and let $1\leq k\leq d$. A {\em $k$-plane} of $A$ is  a sub-tensor of $A$ with $k$ free indices, i.e., fixing $d-k$ indices.
 If the sum of the entries of $A$ on every $k$-plane is 1, then $A$ is said to be {\em $k$-stochastic} (see, e.g.,  \cite{BrCs91, WZ17Per}). So,  being line-stochastic is
 1-stochastic; being plane-stochastic is 2-stochastic.

All $k$-stochastic nonnegative tensors of order $d$ and dimension $n$ comprise
a polytope (if not empty) in $\Bbb R^{n^d}$.
The Krein-Milman Theorem ensures  that  every polytope
is the convex hull of its extreme points.
It is an interesting and unsolved problem (see, e.g., \cite{LZZ17}) to determine the extreme points
  for a polytope of stochastic tensors. For general tensors and their properties, the reader is referred to the books \cite{ WeiBook16, QLbook17}.

  In Section \ref{Sec:Preli}, we start with some properties of extreme points of a polytope that we will use in Section \ref{Sec:NewBounds} and state our problems in the setting of linear programming. In section \ref{Sec:LPBounds},  we use the approaches from
  linear programming and game theory to provide new upper bounds. In Section \ref{Sec:equations}, we study and characterize  the  polytopes $\mathcal{L}_n$ and $\mathcal{P}_n$, finding the dimensions of the polytopes and the ranks of the coefficient matrices. In Section \ref{Sec:NewBounds}, we present a sharper lower bound
  for $f_0(\mathcal{L}_n)$ and new bounds for $f_0(\mathcal{P}_n)$.  In the last section, Section~\ref{Sec:LP}, we show a result about the relation of the extreme points of
  $\mathcal{L}_n$ and $\mathcal{P}_n$ and conclude the paper with a question for further studies.

\section{Preliminaries}\label{Sec:Preli}

If $A=(a_{ij})$ is an $n\times n$ doubly stochastic matrix, then the matrices
$A_1=(a_{i\a(j)}
)$, $A_2=(a_{\a(i)\a(j)}
)$,   and $A_3=(a_{\a(i)\b(j)}
)$ are also  doubly stochastic matrices for all $\a, \b \in S_n$, where $S_n$ is the symmetric group of order $n$. Moreover,
$A$ is a permutation  matrix  if and only if $A_1,$ $ A_2,$ and $ A_3$ are permutation  matrices.
That is,  permutation transformations preserve extremal properties of extremal matrices.
Inductively, for tensors of higher orders,  we have the following.

\begin{pro} Let  $A=(a_{i_1i_2\cdots i_d})$ be a nonnegative tensor of order $d$ and dimension $n$. Then
 $A$ is $k$-stochastic if and only if $A^{\pi}:=(a_{\pi_1(i_1)\pi_2(i_2)\cdots \pi_d(i_d)})$ is
 $k$-stochastic; and $A$ is extremal (for  the polytope)  if and only if $A^{\pi}$ 
 is extremal, where
 $\pi_1, \pi_2, \dots, \pi_d\in S_n$. In particular,
  an $n\times n\times n$ line- or plane-stochastic tensor $A=(a_{ijk})$ is extremal if and only if $A^{\pi}=(a_{\a(i)\b(j)\g(k)})$ and
  $A^{\pi}=(a_{\a(i)\a(j)\a(k)})$ are extremal for all $\a, \b, \g\in S_n$.
\end{pro}

The second part of the proposition has appeared in
\cite[Corollary 2.6]{KLX16}.

Let $A$ be an $n\times n\times n$ stochastic tensor with pages $A_1, A_2, \dots, A_n$;
each  $A_i$ is doubly stochastic, $i=1, 2, \dots, n$.
Through flattening, we can simply write $A=[A_1, A_2, \dots, A_n]$.
Let $P$ be an $n\times n$ permutation matrix. Denote
$$P^tAP=[P^tA_1P, P^tA_2P, \dots, P^tA_nP] \quad (\mbox{$t$ for transpose}).$$

\begin{pro}  $A$ is a vertex  if and only if $P^tAP$ is a vertex   for any $P$.
\end{pro}

\proof
  $A$ is not extremal if and only if  $A=r E+(1-r)F$, where $0<r <1$ and $E\not =F$   in
$\mathcal{L}_n$ (or $\mathcal{P}_n$), if and only if  $P^tAP=r P^tEP+(1-r)P^tFP$.
\qed

By a (0,1)-tensor we mean a tensor in which every entry is either 0 or 1.  A line-permutation (tensor) is a (0,1)-tensor in which every line has exactly one 1, and a plane-permutation is a (0,1)-tensor in which every plane has exactly one 1. (A $k$-permutation is
a (0,1)-tensor in which every $k$-plane has exactly one 1.)

 For $n=2, 3$,
$\mathcal{L}_n$ and $\mathcal{P}_n$ have been well studied. For example,
$\mathcal{L}_3$ has 66 vertices, 12 of which are (0,1)-type and 54 are non-(0,1)-type, containing $\frac12$'s. Moreover,
the number of positive entries in any extremal tensor of
  $\mathcal{L}_n$ is no more than  $3n^2-3n+1$, and
the number of positive entries in any extremal tensor  of $\mathcal{P}_n$ is no more than  $3n-2$ (see, e.g., \cite{FiSw85, JurRys68}).

The triply line-stochastic tenors form the polytope in $\Bbb R^{n^3}$
\begin{equation}\label{LEq}
\mathcal{L}_n=\{ A=(a_{ijk}) \mid  a_{ijk}\geq 0, \; 1\leq i, j, k\leq n \},
 \end{equation}
 where
\begin{equation}\label{LEqs}
\begin{cases}
\vspace{.06in}
\sum_{i=1}^n a_{ijk}=1, & \quad 1\leq  j, k\leq n,  \\
\vspace{.06in}
\sum_{j=1}^n a_{ijk}=1, & \quad 1\leq  i, k\leq n, \\
\vspace{.06in}
 \sum_{k=1}^n a_{ijk}=1, & \quad 1\leq  i, j\leq n
 \end{cases}
 \end{equation}
and the triply plane-stochastic tensors form the polytope in $\Bbb R^{n^3}$
 \begin{equation}
\mathcal{P}_n=\{ A=(a_{ijk}) \mid  a_{ijk}\geq 0, \; 1\leq i, j, k\leq n \}
 \end{equation}
 where
\begin{equation}\label{PEqs}
\begin{cases}
\vspace{.06in}
\sum_{i, j=1}^n a_{ijk}=1, & \quad 1\leq  k\leq n,  \\
\vspace{.06in}
\sum_{j, k=1}^n a_{ijk}=1, & \quad 1\leq  i\leq n,  \\
\vspace{.06in}
 \sum_{i, k=1}^n a_{ijk}=1, & \quad 1\leq  j\leq n.
 \end{cases}
 \end{equation}

 Extremal values of certain functions defined on $\mathcal{L}_n$ and $\mathcal{P}_n$ with integral constraints for (\ref{LEqs}) and  (\ref{PEqs})
(i.e., $a_{ijk}\in \{0, 1\}$),
 are known as multi-index assignment problems in linear programming and optimization; and they are
 $\mathcal{NP}$-hard (see, e.g., \cite[pp.\,9--11 and Chapter 10]{BDM09}).

Viewing each $A=(a_{ijk})\in \mathcal{L}_n$ as an element of $\Bbb R^{n^3}$,
  we write the equations in  (\ref{LEqs}) in a  linear equation system form:
 \begin{equation}\label{Eq:L}
 Lx=l,
 \end{equation}
  where
 $L$ is a $3n^2 \times n^3$ (0,1)-matrix,
 $x\in \Bbb R^{n^3}$ is a nonnegative solution (i.e., $x_{ijk}\geq 0$, arranged in lexicographic order), and
 $l=(1, \dots, 1)^t\in \Bbb R^{3n^2}$.

We do the same for the equations in  (\ref{PEqs}) to get
 \begin{equation}\label{Eq:P}
 Px=p,
 \end{equation}
  where
 $P$ is a $3n \times n^3$ (0,1)-matrix,
  $x\in \Bbb R^{n^3}$ is a nonnegative solution (i.e., $x_{ijk}\geq 0$, arranged in lexicographic order), and
 $p=(1, \dots, 1)^t\in \Bbb R^{3n}$.

Extreme points of $\mathcal{L}_n$ and $\mathcal{P}_n$ are   optimal solutions of (\ref{Eq:L}) and (\ref{Eq:P}).

\section{Bounds for $f_0(\mathcal{L}_n)$}\label{Sec:LPBounds}

Let $f_0(\mathcal{L}_n)$ be the number of vertices of $\mathcal{L}_n$.
Estimation of $f_0(\mathcal{L}_n)$  has been witnessed
 in three  ways:
{(1)}.  Combinatorial method using Latin squares.
Ahmed,  De Loera, and  Hemmecke (see \cite[Theorem 2.0.10]{Mayathesis04} or \cite[Theorem~0.1]{Mayathesis03})
 gave an explicit lower bound $\frac{(n!)^{2n}}{n^{n^2}}$. This lower bound is immediately
 superiorized by the one obtained by Latin squares, because the number of Latin squares of order $n$, denoted by $L(n)$, is equal to the number of
 $n\times n\times n$ line-stochastic (0,1)-tensors (see \cite{JurRys68} or \cite[pp.\,159-161]{Lint92}). Note that every (0-1)-stochastic tensor is
 an extreme point.
{(2)}.
Analytic and topological approach by using hyperplane and induction. Chang, Paksoy, and Zhang \cite{ChangPZ16} showed an upper bound (see Theorem~\ref{CPZ16} below). {(3)}. Computational geometry approach using the known results, i.e., the Lower and Upper Bound Theorems,  on polytopes.
Li, Zhang and Zhang
\cite{LZZ17}
 presented  that the upper bound obtained in this way is better (shaper) than the previous one
 in \cite{ChangPZ16}. However, the lower bound is no better.

 It has remained as an open question whether a bound (``good" or ``bad") can be obtained through optimization or linear programming. Taking this approach,
  we show new upper bounds obtained via this method. It appears to us that comparing the new bounds with the existing bounds (see Theorem~\ref{LZZ17} below) is an uneasy task; so we leave it for further investigation.

\begin{thm}
 [Chang, Paksoy, and Zhang 2016 \cite{ChangPZ16}]\label{CPZ16}
 Let $f_0(\mathcal{L}_n)$ be the number of vertices of the polytope $\mathcal{L}_n$ of
 the $n\times n\times n$ line-stochastic tensors.
Then
$$\frac{(n!)^{2n}}{n^{n^2}}\leq f_0(\mathcal{L}_n)\leq \frac{1}{n^3} \cdot {p(n)\choose n^3-1}, \;\; \mbox{where $p(n)=n^3+6n^2-6n+2$}.$$
\end{thm}

\begin{thm}[Li, Zhang, and Zhang 2017 \cite{LZZ17}]\label{LZZ17}
Let $f_0(\mathcal{L}_n)$ be the number of vertices of the polytope $\mathcal{L}_n$ of
 the $n\times n\times n$ line-stochastic tensors. Then
$$L(n)\leq f_0(\mathcal{L}_n)\leq \left (\hspace{-.08in} \begin{array}{c}
n^3- \lfloor \frac{(n-1)^3+1}{2}\rfloor\\
 3n^2-3n+1
 \end{array} \hspace{-.08in} \right )+\left (\hspace{-.08in}  \begin{array}{c}
n^3- \lfloor \frac{(n-1)^3+2}{2}\rfloor\\
 3n^2-3n+1
 \end{array} \hspace{-.08in} \right  ).
$$
\end{thm}

To proceed, we cite a result from linear programming (see, e.g., \cite[p.\,73]{Gass03} or
\cite[pp.\,98--108]{MarDover13}).  The vertices of a convex set are characterized as follows.

\begin{lem}\label{lem:LP}
Let $K$ be the convex set  $K=\{x\in \Bbb R^n \mid Ax=b\}$, where
$A$ is $m\times n$ and $b\in \Bbb R^m$. Then $x$ is a vertex of $K$ if and only if $x$ is a
(feasible) basic solution to $Ax=b$; that is, $x$ is a vertex of $K$ if and only if the columns of $A$ corresponding to nonzero components of $x$ are linearly independent.
\end{lem}

With this result, we present an upper bound for $f_0(\mathcal{L}_n)$.

\begin{thm}\label{ZZ18OR}
Let $f_0(\mathcal{L}_n)$ be the number of vertices of the polytope $\mathcal{L}_n$ of the
 $n\times n\times n$ line-stochastic tensors.  Then
 $$f_0(\mathcal{L}_n)\leq \sum_{k=n^2}^{3n^2-3n+1} {{n^3}\choose {k}}.$$
\end{thm}

\proof
The matrix $L$ in (\ref{LEqs}) $Lx=l$ is $3n^2\times n^3$.
If $x$ is a vertex of $\mathcal{L}_n$, then the columns of $L$ corresponding to the nonzero components of $x$ are linearly independent.  Note that $x$ has at least $n^2$ positive components
 and at most $3n^2-3n+1$ positive components. We claim that for different vertices $x_1$, $x_2$, say, the corresponding sets of linearly independent columns of $L$ are different. Suppose otherwise that $M$ consists of the columns for $x_1$ and $x_2$. Then $Mx_1=l$ and $Mx_2=l$. Since $M$ is left-invertible, we have a matrix $N$ such that $NM=I$. So,
 $x_1=Nl$ and $x_2=Nl$, which result in $x_1=x_2$, a contradiction.

 Considering a vertex with $n^2$ positive components, then
  there exist $n^2$ columns of $L$ that are linearly independent. There are
 at most  ${{n^3}\choose {n^2}}$ selections of such  columns. So,  there are
 at most ${{n^3}\choose {n^2}}$ vertices with $n^2$ positive components.
 Likewise, for each $k$, $n^2\leq k \leq 3n^2-3n+1$, if there exists a vertex with
 $k$ positive components, then
 there are
 at most  ${{n^3}\choose {k}}$ selections of linearly independent  columns corresponding to the vertices with
 $k$ positive components. Therefore, there are at most  ${{n^3}\choose {k}}$ vertices with
 $k$ positive components. It follows that
 $$f_0(\mathcal{L}_n)\leq
 {{n^3}\choose {n^2}}+ \cdots + {{n^3}\choose {3n^2-3n+1}}.  \qed $$

Another characterization of extreme points of polytopes that came to our attention is the following one from game theory (see, e.g., \cite[p.\,84]{Bur59}).

\begin{lem}\label{OpLem}
Let $\P$ be a polytope in $\Bbb R^d$ given by the system of linear inequalities
$$Bx\geq b, \quad \mbox{where $B$ is $m\times d$, $b\in \Bbb R^m$,  and $x\in \Bbb R^d$}.$$
 Then
a point $x_0\in \Bbb R^d$ is an extreme point of $\P$ if and only if  $B$ has $d$ linearly independent rows
for which equalities in $Bx_0\geq b$ hold.
\end{lem}

\begin{thm}\label{ZZ18}
Let $f_0(\mathcal{L}_n)$ be the number of vertices of the polytope $\mathcal{L}_n$ of the
 $n\times n\times n$ line-stochastic tensors. Then
 $$f_0(\mathcal{L}_n)\leq {{n^3+3n^2-3n+1}\choose {n^3}}.$$
\end{thm}

\begin{proof}
The polytope $\mathcal{L}_n\subset \Bbb R^{n^3}$ is defined by the linear inequalities
(\ref{LEq}) and equalities (\ref{LEqs}). There are $3n^2$ equations in (\ref{LEqs}). We can reduce
these equations to fewer equivalent equations . Of the $n^3$ variables $a_{ijk}$, $(n-1)^3$ of them, say $a_{ijk}$, $1\leq i, j, k\leq n-1$,  are free, the rest
$a_{ijk}$'s with an index $n$ are dependable. It turns out there are
$3n^2-3n+1$ equations (see, e.g., \cite[p.\,182]{FiSw85}). With the $n^3$ inequalities
$a_{ijk}\geq 0$, we  can write the polytope $\mathcal{L}_n$ in the form
$\{x\mid Bx \geq b\}$, where $B$ is an $(n^3+3n^2-3n+1)\times n^3$ matrix.

Set
  $d=n^3$ in Lemma~\ref{OpLem}.  If $x_0$ is an extreme point of $\mathcal{L}_n$, then
  $B_0x_0=b_0$, where $B_0$ consists of $d$ linearly independent rows of $B$ and $b_0$ is the vector of the components of $b$ corresponding to these rows of $B_0$. It follows that
  $B_0$ is a square matrix and it is nonsingular. Thus $x_0$ is uniquely determined by
  $B_0$ and $b_0$.  It turns out there are at most ${{n^3+3n^2-3n+1}\choose {n^3}}$ extreme points for $\mathcal{L}_n$. \qed
\end{proof}




We point out that Linial and Luria presented an estimate of $f_0(\mathcal{L}_n)$ in \cite{LL14}, showing that the polytope $\mathcal{L}_n$ 
 has at least as many vertices  as
\begin{equation*}\label{LL}
 \Big ( \big (1+o(1)\big )\frac{n}{e^2}\Big )^{n^2\cdot \big (\frac32-o(1)\big )}.
\end{equation*}

\section{Dimensions of the polytopes and matrix ranks}\label{Sec:equations}
The definition of a polytope comes in a few different but equivalent forms, one of which is that it is a bounded set enclosed by closed halfspaces. Such a definition is intuitive and geometric. In this section, we discuss the polytopes $\mathcal{L}_n$ (triply line-stochastic tensors) and $\mathcal{P}_n$ (triply plane-stochastic tensors), writing them as intersections of closed halfspaces, determining the dimensions of the polytopes, and finding the ranks of the corresponding coefficient matrices $L$ and $P$.


Let $\mathcal{H}$ denote a generic hyperplane $\{x : \ell(x)=c\}$, where $\ell(x)$ is a linear form in $x$ and $c$ is a constant. Let
$\mathcal{H}^+=\{x : \ell(x)\geq c\}$ and $\mathcal{H}^-=\{x : \ell(x)\leq c\}$. Then $\mathcal{H}^+$ and $\mathcal{H}^-$ are closed halfspaces.
Define  the hyperplanes:
$$\mathcal{A}_{ijk}: \;\; a_{ijk}= 0\;\; \mbox{(e.g., $\mathcal{A}_{123}$ for $a_{123}=0$)}$$
\vspace{-.1in}
$$\mathcal{H}_{ \cdot jk}: \quad
\sum_{i=1}^n a_{ijk}=1, \quad \forall j, k,$$
\vspace{-.1in}
$$\mathcal{H}_{ i\cdot k}: \quad
\sum_{j=1}^n a_{ijk}=1, \quad \forall i, k,$$
\vspace{-.1in}
$$\mathcal{H}_{ ij\cdot}: \quad
 \sum_{k=1}^n a_{ijk}=1, \quad \forall i, j.
$$
We have
  $$\mathcal{L}_n=\cap_{ijk}\mathcal{A}^+_{ijk}\cap \H_{ \cdot jk}^+\cap \H_{ \cdot jk}^-
  \cap \H_{ i\cdot k}^+\cap \H_{ i\cdot k}^-\cap \H_{ ij\cdot}^+\cap \H_{ ij\cdot}^- $$
 Thus,
 $\mathcal{L}_n$ is a polytope (in $\Bbb R^{n^3}$) enclosed by closed halfspaces. Of
  the $n^3$ variables in equations (\ref{LEqs}), $3n^2-3n+1$ are not free (dependent of others)
  due to the constraints.  We see that the dimension of the polytope $\mathcal{L}_n$ is $\dim (\mathcal{L}_n)=n^3-3n^2+3n-1=(n-1)^3$. It follows that the rank of the
   coefficient matrix $L$ in (\ref{Eq:L}) is $3n^2-3n+1$.
The  polytope $\mathcal{L}_n$  has $n^3$ facets because every
 $\mathcal{A}_{ijk}$ is a supporting hyperplane due to the fact that $\mathcal{L}_n\subseteq \mathcal{A}_{ijk}^+$ and
 $ \mathcal{A}_{ijk}\cap \mathcal{L}_n\not = \emptyset$.
 Note that
  $\dim (\mathcal{A}_{ijk}\cap \mathcal{L}_n)= \dim (\mathcal{L}_n)-1$ since one position (index or axis)  is set to 0.
 In contrast, finding the number of vertices (i.e., faces of dimension 0) for
 $\mathcal{L}_n$ is a very difficult open question. We put these together as a theorem.

\begin{thm} The dimension of the polytope $\mathcal{L}_n$ is $(n-1)^3$.
Every $n\times n\times n$ line-stochastic tensor can be expressed as a convex combination of at most $n^3-3n^2+3n$ vertices of $\mathcal{L}_n$. The rank of the matrix $L$ in equation {\rm (\ref{Eq:L})} is $3n^2-3n+1$.
\end{thm}

In a similar way one can write the polytope  $\mathcal{P}_n$ of the triply plane-stochastic
tensors as a finite intersection of closed halfspaces via hyperplanes.
The dimension of $\mathcal{P}_n$ is
$n^3-3n+2$.
Since $\mathcal{P}_n\subseteq \mathcal{A}_{ijk}^+$,
 $\mathcal{F}_{ijk}:=\mathcal{A}_{ijk}\cap \mathcal{P}_n\not = \emptyset$, and
  $\dim \mathcal{F}_{ijk} = \dim (\mathcal{P}_n)-1$,
   $\mathcal{F}_{ijk}$ are the facets of $\mathcal{P}_n$ and $\mathcal{P}_n$  has $n^3$ facets.

\begin{thm} The dimension of the polytope $\mathcal{P}_n$ is $n^3-3n+2$.
Every $n\times n\times n$ plane-stochastic tensor can be expressed as a convex combination of at most $n^3-3n+3$ vertices of $\mathcal{P}_n$. The rank of the matrix $P$ in equation {\rm (\ref{Eq:P})} is $3n-2$.
\end{thm}

\section{New bounds}\label{Sec:NewBounds}
In their seminal paper \cite{{JurRys68}} on configurations
and decompositions of multidimensional arrays, Jurkat and Ryser demonstrated a bijection  between the Latin squares of order $n$ and the  (0,1)-permutation tensors of 3rd order and dimension $n$
(see also, e.g., \cite[p.\,159]{Lint92}).
Since every (0,1)-line-permutation tensor is an extreme point of polytope $\mathcal{L}_n$,
it follows immediately that $f_0(\mathcal{L}_n)$ is bounded below by $L(n)$, the number of Latin squares.  This section is to present a new
lower bound for $f_0(\mathcal{L}_n)$ and also to give lower and upper bounds for
$f_0(\mathcal{P}_n)$.

\begin{thm} Let $L(n)$ be the number of Latin squares of order $n$ and let
$f_0(\mathcal{L}_n)$ be the number of vertices of $\mathcal{L}_n$ (the triply line-stochastic tensors). Then
$$(5n^3-9n^2-2n+3)+ L(n)\leq f_0(\mathcal{L}_n).
$$
\end{thm}

\begin{proof}
Let $E$ be a vertex of $\mathcal{L}_n$ having $k$ zero entries.
It is known that $k> n^3-3n^2+3n-1=(n-1)^3$ (see, e.g., \cite{FiSw85}). Let
$e_{\max}$ be the largest nonzero entry of $E$. By the discussions in Section~\ref{Sec:Preli},
we can place $e_{\max}$  in any position of the $k$ zeros through transformations (or permutations).  In this way, we can generate $k$ extreme points from $E$.

If $n=1, 2, 3$, we can easily verify the inequalities. For $n=4$, there are 576 Latin squares, i.e., $L(4)=576$, while $f_0(\mathcal{L}_4)$ is much bigger (see, e.g., \cite{KLX16}).

Assume $n\geq 5$.
Collecting the extreme points constructed by Fischer and Swart \cite{FiSw85}, we have the following types of extreme points for $\mathcal{L}_n$:
\begin{itemize}
\item[$E_{1}$]: There are $2n^2$ nonzero entries $\frac12$; other entries are 0. ($n^3-2n^2$ 0's.) 
\item[$E_{2}$]: There are $2n^2-n$ nonzero entries; $e_{\max}=1$. ($n^3-2n^2+n$ 0's.)  
\item[$E_{3}$]: There are $n^2+3n-1$ nonzero entries; $e_{\max}=1$. ($n^3-n^2-3n+1$ 0's.)  
\item[$E_{4}$]: There are $n^2+3n-3$ nonzero entries; $e_{\max}=1$. ($n^3-n^2-3n+3$ 0's.)  
\item[$E_{5}$]: There exists a minimal positive entry $\frac{1}{n-1}$. (At least $(n-1)^3$ 0's.)  
\end{itemize}

For $E_1$ type extreme points,  we can place $\frac12$ in any of those  $n^3-2n^2$ positions with 0 to generate (at least) $n^3-2n^2$  additional extreme points.
For $E_2$ type extreme points,  we   place $e_{\max}=1$ in any of those  $n^3-2n^2+n$ positions with 0 to generate (at least) $n^3-2n^2+n$  additional extreme points. Do the same thing for $E_3$ and $E_4$.
For $E_5$ type extreme points,  we  place $\frac{1}{n-1}$ in any of those  $(n-1)^3$ positions with 0 to generate (at least) $(n-1)^3$  additional extreme points.
Note that none of the above extreme points is a (0,1)-tensor. By adding, we get
$$ (5n^3-9n^2-2n+3)+L(n)\leq f_0(\mathcal{L}_n). \quad \qed $$
\end{proof}

It is tempting to apply the same idea of placing other nonzero entries in the zero positions to generate more extreme points. We point out, however, this may not work in general because placing the largest entry affects the positions of other entries. For instance, the extreme point obtained by placing the smallest nonzero entry in a zero position may have been already obtained by placing the largest entry in some zero position.

Now we turn our attention to the polytope of plane-stochastic tensors. Recall the famous McMullen Upper Bound Theorem (UBT) \cite{McM70} (see also, e.g., \cite[p.\,90]{Bro83}):
 the number $f_0(\mathcal{P})$ of vertices of a  polytope $\mathcal{P}$ of dimension $d$ with $f_{d-1}$ facets is bounded as follows:
\begin{equation*}\label{upper}
f_0(\mathcal{P)}\leq  { {f_{d-1} - \lfloor \frac {d+1} {2} \rfloor } \choose { f_{d-1}-d}}  +  {{f_{d-1} - \lfloor \frac {d+2} {2} \rfloor } \choose { f_{d-1}-d} }.
\end{equation*}

\begin{thm}\label{thm:PN}
 Let $\mathcal{P}_n$ be the polytope of triply plane-stochastic tensors. Then
$$(n!)^{2}\leq f_0(\mathcal{P}_n)\leq
\left (\hspace{-.08in} \begin{array}{c}
  \frac{n^3+3n-2}{2}\\
 3n-2
 \end{array} \hspace{-.08in} \right )+\left (\hspace{-.08in}  \begin{array}{c}
 \frac{n^3+3n-4}{2}\\
 3n-2
 \end{array} \hspace{-.08in} \right  ).
$$
\end{thm}

\proof  The left-hand side inequality is obtained by observing that every (0,1)-plane-stochastic tensor is an extreme point of $\mathcal{P}_n$. There are
$(n^2)(n-1)^2\cdots 2^2=(n!)^2$ such  tensors.
The
 inequality on the right-hand side is due to the McMullen UBT with $d=n^3-3n+2$ and
$f_{d-1}=n^3.$  \qed

Using Lemma \ref{lem:LP}, we can obtain an analogous result of
Theorem \ref{ZZ18OR} for $\mathcal{P}_n$:
\begin{equation}\label{eq:UpperP}
f_0(\mathcal{P}_n)\leq \sum_{k=n}^{3n-2} {n^3\choose k}.
\end{equation}

  \section{Line and plane extreme stochastic tensors}\label{Sec:LP}

  Consider the tensors ${A}$ in $\mathcal{L}_n$. Since $A\in \mathcal{L}_n$ is triply line-stochastic, we see that $\frac{1}{n} A$ is triply plane-stochastic, i.e., $\frac{1}{n} A\in \mathcal{P}_n$ . It is known that every (0,1)-permutation tensor in $\mathcal{L}_n$ is an extreme point of $\mathcal{L}_n$. A natural question would be:
can one obtain an extreme point of $\mathcal{P}_n$ by
\lq\lq scaling" the (0,1)-permutation tensors in $\mathcal{L}_n$? The answer is negative if $n\ge 3$. For $n=2$, let $A=(a_{ijk})$ with
$$a_{111}=1,\, a_{121}=0,\, a_{211}=0, a_{221}=1, \,
a_{112}=0, \,a_{122}=1, \,a_{212}=1, \,a_{222}=0.$$
Then $A$ is one of the two vertices of $\mathcal{L}_2$; $\frac12 A$ is one of the six
vertices of $\mathcal{P}_2$.

  \begin{pro}\label{Pro12} Let $n\ge 3$.
   If $A$ is an extreme point of $\mathcal{L}_n$ with (0,1)-entries, then
  $\frac{1}{n} A$ is not an extreme point of $\mathcal{P}_n$.
  \end{pro}

  \proof Let $A\in \mathcal{L}_n$ be a (0,1)-permutation tensor for $n>2$. Then
  $\frac{1}{n} A$ contains $n^2$ nonzero entries (i.e., 1's). However, any extreme point of
   $\mathcal{P}_n$ cannot have more than $3n-2$ positive entries (see, e.g., \cite{BrCs75laa}).
  \qed

We conclude the paper with a question regarding Proposition \ref{Pro12}: How about other non-(0,1) extreme points of  $\mathcal{L}_n$? That is, is it possible that
 some $A$ is an extreme point of $\mathcal{L}_n$ and  $\frac{1}{n} A$ is an extreme point of $\mathcal{P}_n$? If so, then what is its structure?

\bigskip
\noindent
{\bf Acknowledgement.}
  The authors thank Chi-Kwong Li for his comments in the early stage of the project.  Fuzhen Zhang thanks the SKKU Applied Algebra \& Optimization Research Center of South Korea for the hospitality during the May 2017 Workshop on Matrix/Operator Theory .
   Fuzhen Zhang's work
 was partially supported by an NSU PFRDG Research Scholar grant and by National Natural Science Foundation of China (NNSF) No.\,11571220 via Shanghai University. Xiao-Dong Zhang's work
  was partially supported by NNSF No.\,11531001, No.\,11271256 and NSFC-ISF Research Program (No. 11561141001).

\end{document}